\documentclass{amsart}
\usepackage{amsmath, amsthm, amsfonts, amssymb}
\usepackage{mathrsfs}
\usepackage{microtype}
\usepackage[utf8]{inputenc}
\providecommand{\noopsort[1]{}}
\usepackage{bbm}
\usepackage{dsfont}
\usepackage{ifthen}
\usepackage{nicefrac}
\numberwithin{equation}{section}
\usepackage{a4}
\usepackage[active]{srcltx}
\usepackage{verbatim}
\usepackage{hyperref}
\usepackage{color}

\usepackage[shortlabels]{enumitem}
\setlist{leftmargin=*}
\setlist[1]{labelindent=1.2\parindent}

\allowdisplaybreaks

\newtheorem{thm}{Theorem}[section]
\newtheorem{coro}[thm]{Corollary}
\newtheorem{prop}[thm]{Proposition}
\newtheorem{lm}[thm]{Lemma}

\theoremstyle{remark}
\newtheorem{rmk}[thm]{Remark}

\newtheorem{examp}[thm]{Example}

\newcommand{\coloneqq}{\mathrel{\mathop:}=}

\renewcommand{\Re}{{\rm Re}\,}

\newcommand{\eps}{\varepsilon}

\newcommand{\sg}{\emph{sign}}

\newcommand{\R}{\mathds{R}}
\newcommand{\C}{\mathds{C}}
\newcommand{\K}{\mathds{K}}
\newcommand{\N}{\mathds{N}}

\newcommand{\tnorm}[1]{{|\kern-0.3ex|\kern-0.3ex|#1|\kern-0.3ex|\kern-0.3ex|}}

\begin{document}
\title{Generation of semigroup for symmetric matrix Schr\"odinger operators in $L^p$-spaces}
\author{A. Maichine}
\address{Dipartimento di Matematica, Universit\`a degli Studi di Salerno, Via Giovanni Paolo II 132, I-84084 Fisciano (SA), Italy}
\email{amaichine@unisa.it}
\date{}
\begin{abstract}
In this paper we establish generation of analytic strongly continuous semigroup in $L^p$--spaces for the symmetric matrix Schr\"odinger operator $div(Q\nabla u)-Vu$, where, for every $x\in\R^d$, $V(x)=(v_{ij}(x))$ is a semi-definite positive and symmetric matrix. The diffusion matrix $Q(\cdot)$ is supposed to be strongly elliptic and bounded and the potential $V$ satisfies the weak condition $v_{ij}\in L^1_{loc}(\R^d)$, for all $i,j\in\{1,\dots,m\}$. We also characterize positivity of the semigroup and we investigate on its compactness. 
\end{abstract}
\maketitle
\section{Introduction}
Parabolic systems {with unbounded coefficients have become} an interesting topic thanks to their application in the study of several phenomena that come from {several sciences such as} economics, physics, chemistry, etc. 
The semigroup theory allows solving {autonomous linear} parabolic systems by studying the properties of the associated second order differential operator; the so--called vector--valued elliptic operator. In the case of absence of a drift term, one obtains a Schr\"odinger operator with matrix potential which we call, in this paper, matrix Schr\"odinger operator. Such operators have the general form $\mathcal{A}=div(Q\nabla\cdot)-V$, where $Q$ and $V$ are, respectively, the diffusion and potential matrices. The matrix Schr\"odinger operator appears, in non--relativistic mechanical quantum, as the Hamiltonian for a system of interacting adsorbate and substrate atoms. The entries of the potential matrix $V$ represent the interparticle electrical
interactions; namely, electron--electron repulsions, electron--nuclear attractions
and nuclear--nuclear repulsions, see \cite{WTA04} and \cite{WRL08}.

{Recently, in \cite{KLMR}, the authors have considered a matrix Schr\"odinger operator of type $\mathcal{A}$ and they have shown, by application of a noncommutative version of the Dore-Venni theorem,  the generation of {a} semigroup in $L^p(\R^d,\R^m)$, $p\in(1,\infty)$, 
under smoothness  and growth assumptions on $Q$ and $V$. 
Further properties of the semigroup like compactness and positivity have been investigated. Afterward, in \cite{Mai-Rha}, similar results as in \cite{KLMR} have been obtained for matrix potentials with diagonal entries of polynomial growth. Moreover, kernel estimates for the associated semigroup have been investigated and the asymptotic distribution of eigenvalues of the matrix Schr\"odinger operator has been established.}
  
 In this article we associate a sesquilinear form, in $L^2(\R^d,\C^m)$, to the matrix Schr\"odinger operator $\mathcal{A}u=(div(Q\nabla u_j))_j-Vu$, $u=(u_1,\dots,u_m)$, where $V=(v_{ij})_{1\le i,j\le m}$ is a symmetric matrix-valued function. As in the scalar case, see \cite[Section 1.8]{Davies}, and under the weakest condition $v_{ij}\in L^1_{\emph{loc}}(\R^d)$, $i,j\in\{1,\dots,m\}$, we prove that $\mathcal{A}$ admits a dissipative self-adjoint realization in $L^2(\R^d,\C^m)$ and we extrapolate its associated semigroup to the spaces $L^p(\R^d,\C^m)$, using a vectorial version of 'Beurling-Denny' criterion of $L^\infty$-contractivity. We also investigate on some properties of the semigroup.
 
  This article is structured as follow: In section 2 we study the associated form to $\mathcal{A}$ and show that $\mathcal{A}$ has a self-adjoint realization $A$ that generates an analytic strongly continuous semigroup $\{T(t)\}_{t\ge 0}$ in $L^2(\R^d,\C^m)$. In section 3, we apply a 'Beurling-Denny' criterion type, see \cite[Theorem 2]{Ouhabaz 99}, to establish $L^\infty$-contractivity of the semigroup $\{T(t)\}_{t\ge 0}$ and thus extrapolate it to $L^p(\R^d,\C^m)$. Section 4 is devoted to characterize positivity, study compactness of the semigroup and analyze the spectrum of $A$. 
  \subsection*{Notation} Throughout this paper we adopt the following notation: $d,m\in\N$, $\K=\R$ or $\K=\C$, $\langle\cdot,\cdot\rangle$ the inner-product of $\K^j$, $j=d,m$.  
  $L^p(\R^d,\K^m)$, $1<p<\infty$, denotes the vectorial Lebesgue space endowed with the norm $$\|\cdot\|_p: f=(f_1,\dots,f_m)\mapsto\|f\|_p\coloneqq \left(\int_{\R^d}(\sum_{j=1}^{m}|f_j|^2)^{\frac{p}{2}}dx\right)^{\frac{1}{p}}.$$
  $H^1(\R^d)$ refers to the classical Sobolev space of order $1$ over $L^2(\R^d)$. $H^1(\R^d,\R^m)$ is the vectorial Sobolev space constitueted of vectorial function $f=(f_1,\dots,f_m)$ such that $f_j\in H^1(\R^d)$, for all $j\in\{1,\dots,m\}$. We note that all the derivatives are considered in the distribution sense. For $y=(y_1,\dots,y_m)\in\R^m$, we write $y\ge 0$ if $y_j\ge 0$ for all $j\in\{1,\dots,m\}$. For $r>0$, $B(r)=\{x\in\R^d : |x|<1 \}$ denotes the euclidean open ball of $\R^d$ of center $0$ and radius $r$. $\chi_E$ is the characteristic function of the set $E$.
\section{Generation of the semigroup in $L^2$}
Throughout we assume the following hypotheses
\subsection{Hypotheses}\label{Hypo}
\begin{enumerate}[(a)]
\item Let $Q:\R^d\to\R^{d\times d}$ be a symmetric matrix-valued function. Assume that there exist positive numbers $\eta_1$ and $\eta_2$ such that
\begin{equation}\label{ellipticity of q}
\eta_1|\xi|^2\le\langle Q(x)\xi,\xi\rangle\le\eta_2|\xi|^2,\qquad x,\xi\in\R^d.
\end{equation}
\item Let $V:\R^d\to\R^{m\times m}$ be a matrix-valued operator such that $v_{ij}=v_{ji}\in L^1_{\emph{loc}}(\R^d)$ for all $i,j\in\{1,\dots,m\}$ and
\begin{equation}\label{accretivity of V}
\langle V(x)\xi,\xi\rangle\ge 0,\qquad x\in\R^d,\;\xi\in\R^m.
\end{equation}
\end{enumerate}
We introduce, for $x\in\R^d$, the inner-product $\langle\cdot,\cdot\rangle_{Q(x)}$ given, for $y,z\in\C^d$, by $\langle y,z\rangle_{Q(x)}:=\langle Q(x)y,z\rangle$ and its associated norm $|z|_{Q(x)}:=\sqrt{\langle Q(x)z,z\rangle}$ for each $z\in\C^d$.
\subsection{The $L^2$-sesquilinear form}
Let us define the sesquilinear form
\begin{equation}\label{explicit form}
a(f,g):=\int_{\R^d}\sum_{j=1}^{m}\langle Q(x)\nabla f_j(x),\nabla g_j(x)\rangle dx+\int_{\R^d}\langle V(x)f(x),g(x)\rangle dx, 
\end{equation}
for $f,g\in D(a)$, where $D(a)$, the domain of $a$, is defined by
\begin{equation}\label{domain of form}
D(a)=\{f=(f_1,\dots,f_m)\in H^1(\R^d,\C^m) :\int_{\R^d}\langle V(x)f(x),f(x)\rangle dx < +\infty \}.
\end{equation}
We endow $D(a)$ with the norm
\begin{align*}
\|f\|_a &=\left(\|f\|_{H^1(\R^d,\C^m)}^2+\int_{\R^d}\langle V(x)f(x),f(x)\rangle dx\right)^{1/2}\\
&=\left(\|f\|_{L^2(\R^d,\C^m)}^2+\sum_{j=1}^{m}\|\nabla f_j\|_{L^2(\R^d,\C^m)}^2+\int_{\R^d}\langle V(x)f(x),f(x)\rangle dx\right)^{1/2}.
\end{align*}
We now give some properties of $a$
\begin{prop}\label{properties of the form}
Assume Hypotheses \ref{Hypo} are satisfied. Then,
\begin{enumerate}
\item [i)] $a$ is densely defined, i.e., $D(a)$ is dense in $L^2(\R^d,\R^m)$.
\item [ii)] $a$ is accretive.
\item [iii)] $a$ is contiuous, i.e., exists $M>0$ such that
\[ |a(f,g)|\leq M\|f\|_a\|g\|_a,\qquad f,g\in D(a). \]
\item [iv)] $a$ is closed i.e., $(D(a),\|.\|_a)$ is a complete space.
\end{enumerate}
\end{prop}
\begin{proof}
\begin{enumerate}
\item [i)] It is obvious that $C_c^\infty(\R^d,\R^m)\subset D(a)$. Hence, $D(a)$ is dense in $L^2(\R^d,\C^m)$.
\item [ii)] \emph{Accretivity:} For $f\in D(a)$ one has
\[\Re\, a(f)=\int_{\R^d}\sum_{j=1}^{m}|Q^{1/2}(x)\nabla f_j(x)|^2 dx+\int_{\R^d}\Re\langle V(x)f(x),f(x)\rangle dx\ge 0.\] 
\item [iii)] \emph{Continuity:} Let $f,g\in D(a)$. By application of Cauchy--Schwartz and Young inequalities one gets
\begin{align*}
|a(f,g)|&\le  \eta_2\sum_{j=1}^{m}\int_{\R^d}|\nabla f_j(x)||\nabla g_j(x)|dx+\int_{\R^d}|\langle V(x)^{1/2}f(x),V(x)^{1/2}g(x)\rangle| dx\\
&\le  \eta_2\sum_{j=1}^{m}\|\nabla f_j\| \|\nabla g_j\|+\left(\int_{\R^d}|V(x)^{1/2}f(x)|^2 dx\right)^{1/2}\left(\int_{\R^d}|V(x)^{1/2}g(x)|^2 dx\right)^{1/2}\\
&\le  \eta_2\left(\sum_{j=1}^{m}\|\nabla f_j\|_{2}^2\right)^{\frac{1}{2}}\left(\sum_{j=1}^{m}\|\nabla g_j\|_{2}^2\right)^{\frac{1}{2}}+\left(\int_{\R^d}\langle V(x)f(x),f(x)\rangle dx\right)\left(\int_{\R^d}\langle V(x)g(x),g(x)\rangle dx\right)\\
&\le  (1+\eta_2)\|f\|_{a}\|g\|_{a}.
\end{align*}
\item [iv)] \emph{Closedness:} Let $(f_n)_{n\in\N}\subset D(a)$ be a Cauchy sequence in $(D(a),\|\cdot\|_a)$. 
Then,
\[ \|f_n-f_l\|_{H^1(\R^d,\R^m)}+\int_{\R^d}\langle V(x)(f_n(x)-f_l(x)),(f_n(x)-f_l(x)\rangle dx\underset{n,l\rightarrow\infty}{\longrightarrow}0, \]
which yields
\[\begin{cases}
f_n-f_l\underset{n,l\rightarrow\infty}{\longrightarrow} 0 \quad in\quad H^1(\R^d,\C^m)\\
\int_{\R^d}|V^{1/2}(f_n-f_l)|^2\underset{n,l\rightarrow\infty}{\longrightarrow} 0
\end{cases}.
\]
Hence, $(f_n)_{n\in\N}$ and $(V^{1/2}f_n)_{n\in\N}$ are Cauchy sequences respectively in $H^1(\R^d,\C^m)$ and $L^2(\R^d,\C^m)$. Therefore
\[\begin{cases}
f_n\longrightarrow f\quad in\quad H^1(\R^d,\C^m)\\
V^{1/2}f_n\longrightarrow g\quad in\quad L^2(\R^d,\C^m)
\end{cases}.
\]
The pointwise convergence of subsequences implies that 
\[ V^{1/2}f=g\in L^2(\R^d,\C^m).\]
Then $f\in D(a)$ and
\[a(f_n-f)=\|f_n-f\|_{H^1(\R^d,\C^m)}^2+\int_{\R^d}|V^{1/2}(x)(f_n-f)(x)|^2 dx\underset{n\rightarrow\infty}{\longrightarrow}0,\]
which ends the proof.
\end{enumerate}
\end{proof}
Now define 
\begin{equation}\label{expression of A}
\mathcal{A}f:=div(Q\nabla f)-Vf=(div(Q\nabla f_j))_{1\le j\le m}-Vf.
\end{equation}
Thanks to proposition \ref{properties of the form} and applying \cite[Proposition 1.22]{Ouha-book} and the well-known Lumer-Phillips theorem, \cite[Chap-II, Theorem 3.15]{Nagel}, one obtains
\begin{coro}\label{coro gen sem in L2}
$\mathcal{A}$ admits a realization $A$ in $L^2(\R^d,\C^m)$ that generates a bounded strongly continuous semigroup $(T(t))_{t\ge 0}$ which is analytic in the open right half plane of $\C$. Moreover, $A$ is selfadjoint and $-A$ is the linear operator associated to the form $a$.
\end{coro}
\begin{rmk}
The form method does not apply for non symmetric potentials. In fact, in \cite[Example 3.5]{KLMR} it has been proved that the semigroup associated to a matrix Schr\"odinger operator with matrix potential 
\[V(x)=\begin{pmatrix}
0 & -x\\
x & 0
\end{pmatrix},\qquad x\in\R,\]
 is not analytic. Otherwise, we show by straighforward computation that the continuity property of the form $a$ fails when one takes instead of a symmetric potential the above antisymmetric one $V$.
Indeed, let $\varphi\in C_c^\infty(\R^d)$ such that $\chi_{B(1)}\le\varphi\le\chi_{B(2)}$. Consider, for $n\ge 1$, $$f_n(x)=\frac{\varphi(x/n)}{\sqrt{1+|x|^2}}e_1\quad \emph{and}\quad g_n(x)=\frac{\varphi(x/n)}{\sqrt{1+|x|^2}}e_2$$ where $\{e_1,e_2\}$ is the canonical basis of $\R^2$. 
Since $V=-V^*$ then $\langle V(x)\xi,\xi\rangle=0$, for all $\xi\in\R^2$. Thus
\[a(f_n)=a(g_n)=\int_{\R}\left|-\dfrac{\varphi(x/n)}{(1+|x|^2)^{\frac{3}{2}}}x+\frac{1}{n}\frac{1}{\sqrt{1+|x|^2}}\nabla\varphi(x/n)\right|^2 dx\]
and
\[|a(f_n,g_n)|=\int_{\R^d}\dfrac{|x|}{(1+|x|^2)}\varphi(x/n)dx.\]
If the continuity property -Proposition \ref{properties of the form} (iii)- of the form were satisfied, then there will exist $C>0$ such that 
\[|a(f_n,g_n)|\le C\|f_n\|_a\|g_n\|_a=C(\|f_n\|_{L^2(\R,\R^2)}^2+a(f_n)).\]
By the Lebesgue dominated convergence theorem one can let $n$ tends to $\infty$ and obtains
\[ \int_{\R}\dfrac{|x|}{1+|x|^2}dx\le C\left(\int_{\R^d}\dfrac{|x|^2}{(1+|x|^2)^3}dx+\int_{\R^d}\dfrac{1}{1+|x|^2}dx\right)<\infty. \]
However, the integral of the left-hand side is infinite.
\end{rmk}
\section{Extension to $L^p$}
In this section we will show that $\mathcal{A}$ has a $L^p$-realization which generates a holomorphic semigroup in $L^p(\R^d,\C^m)$, $1<p<\infty$. In order to do so we prove that, for every $t>0$, the restriction $T(t)_{|L^2\cap L^\infty}$ of $T(t)$ to $L^2(\R^d,\C^m)\cap L^\infty(\R^d,\C^m)$ can be extended to a bounded operator $T_p(t)$ in $L^p(\R^d,\C^m)$, $2<p<\infty$. Then, we show that $(T_p(t))_{t\ge 0}$ is strongly continuous. Moreover, since $(T(t))_{t\ge 0}$ is self-adjoint, the semigroups $(T_p(t))_{t\ge 0}$, for $p\in(1,2)$ will obtained by duality arguments. For this aim it suffices that $(T(t))_{t\ge 0}$ satisfy the following $L^\infty$-contractivity property:
\begin{equation}\label{L^8 contractivity}
\|T_2(t)f\|_\infty\le \|f\|_\infty,\qquad \forall f\in L^2(\R^d,\R^m)\cap L^\infty(\R^d,\R^m).
\end{equation}
A characterisation of \eqref{L^8 contractivity} via the associated form is given by E.M. Ouhabaz in \cite[Theorem2]{Ouhabaz 99}. According to this characterisation, \eqref{L^8 contractivity} holds true if the following is satisfied
\begin{prop}\label{Beurling-Deny condition}
Let $f\in H^1(\R^d,\R^m)$ and set $\sg(f)\coloneqq\frac{f}{|f|}\chi_{\{f\neq 0\}}$. Then,
\begin{enumerate}[(i)]
\item $f\in D(a)$ implies $(1\wedge |f|)\sg(f)\in D(a)$;
\item for every $f\in D(a)$, one has
\[a((1\wedge |f|)\sg(f))\le a(f). \]
\end{enumerate}
\end{prop}
Now we proceed to prove the above proposition. As a first step, we establish the assertion (i) of the above proposition in the following lemma  
\begin{lm}\label{H^1 stable under |.| and P8}
\begin{itemize}
\item[a)] Assume $f\in H^1(\R^d,\R^m)$. Then, $|f|\in H^1(\R^d)$ and
\begin{equation}\label{gradient of |.|}
\nabla |f|=\dfrac{\sum_{j=1}^{m}f_j\nabla f_j}{|f|}\chi_{\{f\neq 0\}}.
\end{equation}
\item[b)] Let $f\in D(a)$. Then $(1\wedge |f|)\sg(f)\in D(a)$. In particular,
\begin{align}
\nabla((1\wedge |f|)\sg(f))_j&=\dfrac{1+\sg(1-|f|)}{2}\frac{f_j}{|f|}\chi_{\{f\neq 0\}}\nabla|f|\\\nonumber
&+\dfrac{1\wedge|f|}{|f|}(\nabla f_j-\frac{f_j}{|f|}\nabla|f|)\chi_{\{f\neq 0\}},
\end{align}
for every $j\in\{1,...,m\}$.
\end{itemize}
\end{lm}
\begin{proof}
a) Let $f\in H^1(\R^d,\R^m)$. Define, for $\varepsilon>0$, $f_\varepsilon=\left(\displaystyle\sum_{j=1}^{m}f_j^2+\varepsilon^2\right)^{\frac{1}{2}}-\varepsilon$. One has 
\[0\le f_\varepsilon=\dfrac{|f|^2}{\left(\displaystyle\sum_{j=1}^{m}f_j^2+\varepsilon^2\right)^{\frac{1}{2}}+\varepsilon}\le |f|. \]
Hence, by dominated convergence theorem $f_\varepsilon\underset{\varepsilon\rightarrow 0}{\longrightarrow} |f|$ in $L^2(\R^d)$. On the other hand, $f_\varepsilon\in H^1_{loc}(\R^d)$ and
\[\nabla f_\varepsilon=\dfrac{\displaystyle\sum_{j=1}^{m}f_j\nabla f_j}{\left(\sum_{j=1}^{m}f_j^2+\varepsilon\right)^{\frac{1}{2}}}\underset{\varepsilon\rightarrow 0}{\longrightarrow}\dfrac{\displaystyle\sum_{j=1}^{m}f_j\nabla f_j}{|f|}\chi_{\{f\neq 0\}}. \]
Again, the dominated convergence theorem yields $|f|\in H^1(\R^d)$ and \eqref{gradient of |.|}.\\ 
b) Let $f\in D(a)$ i.e., $f\in H^1(\R^d,\R^m)$ and $V^{1/2}f\in L^2(\R^d,\R^m)$. One has
\begin{align*}
\int_{\R^d}\langle V(x)(1\wedge |f|)\sg(f),(1\wedge |f|)\sg(f)\rangle dx&\le\int_{\{f\neq 0\}}\left(\dfrac{1\wedge |f|}{|f|}\right)^2\langle V(x)f(x),f(x)\rangle dx\\
&\le \int_{\R^d}\langle V(x)f(x),f(x)\rangle dx<\infty.
\end{align*}
Now remains to show that $(1\wedge |f|)\sg(f)\in H^1(\R^d,\R^m)$. Set
$$Pf:=(1\wedge |f|)\sg(f)=(1\wedge |f|)\frac{f}{|f|}\chi_{\{f\neq 0\}},$$
 and 
 \[P_\varepsilon f:=(1\wedge |f|)\frac{f}{|f|+\varepsilon}=\dfrac{1+|f|-|1-|f||}{2}\frac{f}{|f|+\varepsilon},\]
for $\varepsilon>0$.
One has
\[
\begin{cases}
|P_\varepsilon f|\le (1\wedge |f|)\le |f|\\
P_\varepsilon f\underset{\varepsilon\rightarrow 0}{\longrightarrow} Pf\qquad a.e.
\end{cases},
\]
which implies that $P_\varepsilon f \to Pf$ in $L^2(\R^d,\R^m)$ as $\eps\to 0$.

On the other hand, $P_\varepsilon f\in H^1_{loc}(\R^d,\R^m)$ and
\begin{align*}
\nabla (P_\varepsilon f)_j&=\nabla \left(\dfrac{1+|f|-|1-|f||}{2}\frac{f_j}{|f|+\varepsilon}\right)\\
&= \dfrac{1+|f|-|1-|f||}{2}\left(\frac{\nabla f_j}{|f|+\varepsilon}-\frac{f_j}{(|f|+\varepsilon)^2}\nabla|f|\right)\\
&+\frac{1}{2}\dfrac{f_j}{|f|+\varepsilon}\left(\nabla|f|+\sg(1-|f|)\nabla|f|\right)\\
&=\dfrac{1\wedge |f|}{|f|+\varepsilon}\left(\nabla f_j-\frac{f_j}{|f|+\varepsilon}\nabla|f|\right)\\
&+ \frac{1}{2}\dfrac{f_j}{|f|+\varepsilon}\left(1+\sg(1-|f|)\right) \nabla|f|.
\end{align*}
Hence,
\begin{align*}
\lim_{\varepsilon\rightarrow 0}\nabla (P_\varepsilon f)_j&=
\dfrac{1\wedge |f|}{|f|}\left(\nabla f_j-\frac{f_j}{|f|}\nabla|f|\right)\chi_{\{f\neq 0\}}\\
&+ \frac{1}{2}\dfrac{f_j}{|f|}\left(1+\sg(1-|f|)\right)\chi_{\{f\neq 0\}} \nabla|f|,
\end{align*}
and
\[ |\nabla (P_\varepsilon f)_j|\le \dfrac{1\wedge |f|}{|f|}(|\nabla f_j|+|\nabla|f||)+|\nabla|f||\le |\nabla f_j|+2|\nabla|f||\;\in L^2(\R^d).\]
 By the dominated convergene theorem we conclude that $Pf=\displaystyle\lim_{\varepsilon\to 0}P_\eps f\in H^1(\R^d,\R^m)$, and
\begin{align*}
\nabla (P f)_j&\coloneqq\lim_{\varepsilon\rightarrow 0}\nabla (P_\varepsilon f)_j\\
&=\dfrac{1\wedge |f|}{|f|}\left(\nabla f_j-\frac{f_j}{|f|}\nabla|f|\right)\chi_{\{f\neq 0\}}
+ \frac{1}{2}\dfrac{f_j}{|f|}\left(1+\sg(1-|f|)\right)\chi_{\{f\neq 0\}} \nabla|f|.
\end{align*}
\end{proof}
In the following we state another lemma where we prove (ii) of Proposition \ref{Beurling-Deny condition}  
\begin{lm}
Let $f\in D(a)$. Then
\begin{equation}
a((1\wedge |f|)\sg(f))\le a(f).
\end{equation}
\end{lm}
\begin{proof}
Let $f\in D(a)$. One has,
\begin{align*}
\alpha_f&:=\langle Q\nabla((1\wedge |f|)sg(f)),\nabla((1\wedge |f|)\sg(f))\rangle\\
&=\sum_{j=1}^{m}\left|\dfrac{1\wedge |f|}{|f|}\left(\nabla f_j-\frac{f_j}{|f|}\nabla|f|\right)\chi_{\{f\neq 0\}}
+ \frac{1}{2}\dfrac{f_j}{|f|}(1+\sg(1-|f|))\chi_{\{f\neq 0\}} \nabla|f|\right|_{Q}^2\\
&=\dfrac{(1+\sg(1-|f|))^2}{4}\chi_{\{f\neq 0\}}|\nabla|f||_Q^2+\frac{(1\wedge|f|)^2}{|f|^2}\chi_{\{f\neq 0\}}\sum_{j=1}^{m}|\nabla f_j-\frac{f_j}{|f|}\nabla|f||_Q^2\\
&+ (1+\sg(1-|f|))\frac{1\wedge|f|}{|f|}\chi_{\{f\neq 0\}}\sum_{j=1}^{m}\langle Q\nabla|f|,(\nabla f_j-\frac{f_j}{|f|}\nabla |f|)\rangle f_j\\
&=\dfrac{(1+\sg(1-|f|))^2}{4}\chi_{\{f\neq 0\}}|\nabla|f||_Q^2\\
&+\frac{(1\wedge|f|)^2}{|f|^2}\chi_{\{f\neq 0\}}\left(\sum_{j=1}^{m}\langle Q\nabla f_j,\nabla f_j\rangle+|\nabla|f||_Q^2-\langle \nabla|f|^2,\frac{\nabla|f|}{|f|}\rangle_Q\right)\\
&+ (1+\sg(1-|f|))\frac{1\wedge|f|}{|f|}\chi_{\{f\neq 0\}}\underset{0}{\underbrace{\left(\frac{1}{2}\langle Q\nabla|f|,\nabla|f|^2\rangle-|f|\langle Q\nabla|f|,\nabla|f|\rangle\right)}}\\
&=\dfrac{(1+\sg(1-|f|))^2}{4}\chi_{\{f\neq 0\}}|\nabla|f||_Q^2+\frac{(1\wedge|f|)^2}{|f|^2}\chi_{\{f\neq 0\}}\left(\sum_{j=1}^{m}\langle Q\nabla f_j,\nabla f_j\rangle-|\nabla|f||_Q^2\right)\\
&= \left(\dfrac{(1+\sg(1-|f|))^2}{4}-\frac{(1\wedge|f|)^2}{|f|^2}\right)\chi_{\{f\neq 0\}}|\nabla|f||_Q^2+\frac{(1\wedge|f|)^2}{|f|^2}\chi_{\{f\neq 0\}}\sum_{j=1}^{m}\langle Q\nabla f_j,\nabla f_j\rangle.
\end{align*}
Discussing the cases $|f|<1$, $|f|=1$ and $|f|>1$, one can easily see that $$\dfrac{(1+\sg(1-|f|))^2}{4}-\frac{(1\wedge|f|)^2}{|f|^2}\le 0.$$ 
Thus,
\begin{align*}
\alpha_f 
&\le \frac{(1\wedge|f|)^2}{|f|^2}\chi_{\{f\neq 0\}}\sum_{j=1}^{m}\langle Q\nabla f_j,\nabla f_j\rangle\\
&\le \sum_{j=1}^{m}\langle Q\nabla f_j,\nabla f_j\rangle.
\end{align*}
Integrating over $\R^d$, one gets
\begin{align*}
a_0((1\wedge |f|)\sg(f))&\coloneqq\int_{\R^d}\langle Q\nabla((1\wedge |f|)\sg(f)),\nabla((1\wedge |f|)\sg(f))\rangle dx\\
&\le \int_{\R^d}\sum_{j=1}^{m}\langle Q\nabla f_j,\nabla f_j\rangle dx\coloneqq a_0(f).
\end{align*}
Therefore,
\begin{align*}
a((1\wedge |f|)\sg(f))&= a_0((1\wedge |f|)\sg(f))+\int_{\R^d}\langle V(x)(1\wedge |f|)\sg(f),(1\wedge |f|)\sg(f)\rangle dx\\
&\le a_0(f)+\int_{\R^d}\langle Vf,f\rangle dx= a(f).
\end{align*}
\end{proof}
As a consequence of the statement of Proposition \ref{Beurling-Deny condition}, one gets
\begin{coro}\label{coro semgpe L8-contra}
The semigroup $\{T(t)\}_{t\ge 0}$ is $L^\infty$-contractive.
\end{coro}
Now, we are able to state our main theorem of this section
\begin{thm}
Let $1<p<\infty$ and assume Hypotheses \ref{Hypo}. Then, $\mathcal{A}$ admits a realization $A_p$ in $L^p(\R^d,\C^m)$ that generates a bounded strongly continuous semigroup $(T_p(t))_{t\ge 0}$. Moreover, $(T_p(t))_{t\ge 0}$ is analytic in the sector of angle $\frac{\pi}{p}$ if $p>2$; it is analytic in the sector of angle $\frac{\pi}{p'}$ if $1<p<2$ where $1/p + 1/p' =1$.
\end{thm}
\begin{proof}
Let $2<p<\infty$. According to Corollary \ref{coro gen sem in L2} and Corollary \ref{coro semgpe L8-contra} $\{T(t)\}_{t\ge 0}$ is self-adjoint and $L^\infty$--contactive. Hence, by the Riesz-Thorin interpolation theorem, $\{T(t)\}_{t\ge 0}$ admits a unique bounded extension $\{T_p(t)\}_{t\ge 0}$ to $L^p(\R^d,\C^m)$; this extension is analytic in the sector of angle $\frac{\pi}{p}$, cf. \cite[Theorem 2.9]{fggor10}. Moreover, for every $f\in L^2(\R^d,\R^m)\cap L^\infty(\R^d,\R^m)$,
\[ \|T(t)f-f\|_p\le \|T(t)f-f\|_2^\theta\|T(t)f-f\|_\infty^{1-\theta}\le 2^{1-\theta}\|f\|_\infty^{1-\theta}\|T(t)f-f\|_2^\theta, \]
where $\theta=\frac{2}{p}$. This shows how $\{T_p(t)\}_{t\ge 0}$ is strongly continuous.\\
Concerning the case $1<p<2$, we prove by duality that $\|T(t)f\|_1\le \|f\|_1$, for every $t>0$, and similarly, we obtain an analytic extrapolation of $\{T(t)\}_{t\ge 0}$ which is strongly continuous.
\end{proof}
\begin{rmk}
We can extrapolate the semigroup $\{T(t)\}_{t\ge 0}$ to a strongly continuous one in $L^1(\R^d,\R^m)$. It is an easier consequence of consistency and $L^p$-contractivity of $\{T(t)\}_{t\ge 0}$, see \cite{v92}.
\end{rmk}
\section{Further properties of the semigroup}
In this section we study positivity and compactness of $\{T(t)\}_{t\ge 0}$ and the spectrum of $A$. We start by positivity
\subsection{Positivity}
In this subsection we give necessary and sufficient condition for positivity of the semigroup $\{T(t)\}_{t\ge 0}$. We use the form characterization of the invariance of convex subsets via semigroups. For this purpose we introduce
$$C^+\coloneqq\{f=(f_1,\dots,f_m)\in L^2(\R^d,\R^m): f\ge 0\qquad a.e.\}.$$ $C^+$ is a closed convex subset of $L^2(\R^d,\R^m)$. The projection $P_+$ on $C^+$ is given by
\[P_+ f\coloneqq f^+=(f_j\wedge 0)_{1\le j\le m},\qquad\forall f\in L^2(\R^d,\R^m).\]
One knows that the projection on a closed convex subsets of a Banach space is uniquely defined and it is easy to check that $P_+$, defined above, is the right one for $C^+$.
We recall that $\{T(t)\}_{t\ge 0}$ is positive if, and only if, $f\in L^2(\R^d,\R^m)$ and $f\ge 0$ imply $T(t)f\ge 0$, for every $t>0$.
 The form characterization of positivity is given by \cite[Theorem 3 (iii)]{Ouhabaz 99} as follow
\begin{prop}\label{caracterisation of positivity by forms}
$(T(t))_{t\ge 0}$ is a positive semigroup if, and only if, $f^+\in D(a)$ for all $f\in D(a)$ and $a(f^+,f^-)\le 0$, where $f^-=f-P_+ f=((-f_j)\wedge 0)_{1\le j\le m}$.
\end{prop}
 By application of \cite[Theorem 3 (iii)]{Ouhabaz 99} we get the following characterisation of positivity of $\{T(t)\}_{t\ge 0}$ in term of entries of the potential matrix $V$
\begin{thm}\label{carac. of positivity by entries of V}
The semigroup $(T(t))_{t\ge 0}$ is positive if and only if $v_{ij}\le 0$, for all $i\ne j\in\{1,\dots,m\}$.
\end{thm}
\begin{proof}

Suppose that $\{T(t)\}_{t\ge 0}$ is positive. Let $i\ne j\in\{1,\dots,m\}$ and consider $f=\varphi(e_i-e_j)$ where $0\le \varphi\in C_c^\infty(\R^d)$ to be arbirarly choosen. One has $f^+=\varphi e_i$, $f^-=\varphi e_j$ and $\langle Q\nabla f^+,\nabla f^-\rangle=0$.
Applying Proposition \ref{caracterisation of positivity by forms} one obtains
\[0\ge a(f^+,f^-)=\sum_{k=1}^{m}\int_{\R^d}\langle Q\nabla u^+_k,\nabla f^-_k\rangle dx+\int_{\R^d}\langle Vf^+,f^-\rangle dx=\int_{\R^d}v_{ij}\varphi^2 dx, \]
 for every $0\le \varphi\in C_c^\infty(\R^d)$. This yields $v_{ij}\le 0$ a.e. Conversely, assume that the off-diagonal entries $v_{ij}$, $i\neq j$, are less than or equal to $0$ and let $f\in D(a)$. Let us show first that $f^+\in D(a)$. According to \cite[Lemma 7.6]{Gilb-Tru} one has, $\nabla f_k^+=\chi_{\{f_k>0\}}\nabla f_k$ and  $\nabla f_k^-=\chi_{\{f_k^-<0\}}\nabla f_k$, hence $f^+\in H^1(\R^d,\R^m)$ and $\langle Q\nabla f^+_k,\nabla f^-_k\rangle=0$. On the other hand,
 \begin{align*}
 \langle Vf,f\rangle &= \langle V(f^+ -f^-),(f^+ - f^-)\rangle\\
 &= \langle Vf^+,f^+\rangle+\langle Vf^-,f^-\rangle-2\langle Vf^+,f^-\rangle\\
 &= \langle Vf^+,f^+\rangle+\langle Vf^-,f^-\rangle-2\sum_{i,j=1}^{m}v_{ij}f^+_i f^-_j\\
 &=\langle Vf^+,f^+\rangle+\langle Vf^-,f^-\rangle-2\underset{\le 0}{\underbrace{\sum_{i\ne j}^{}v_{ij}f^+_i f^-_j }}\\
 &\ge \langle Vf^+,f^+\rangle+\langle Vf^-,f^-\rangle\\
 &\ge \langle Vf^+,f^+\rangle.
 \end{align*}
 Thus $\displaystyle\int_{\R^d}\langle Vf^+,f^+\rangle dx\le \displaystyle\int_{\R^d}\langle Vf,f\rangle dx<\infty$. Consequently $f^+\in D(a)$. Moreover,
 \begin{align*}
 a(f^+,f^-)&= \sum_{k=1}^{m}\int_{\R^d}\langle Q\nabla f^+_k,\nabla f^-_k\rangle dx+\int_{\R^d}\langle Vf^+,f^-\rangle dx\\
 &= \int_{\R^d}\sum_{i,j=1}^{m}v_{ij}f^+_i f^-_j dx\\
 &= \int_{\R^d}\sum_{i=1}^{m}v_{ii}f^+_i f^-_i dx+\int_{\R^d}\sum_{i\ne j}v_{ij}f^+_i f^-_j dx\\
 &=\int_{\R^d}\sum_{i\ne j}^{}v_{ij}f^+_i f^-_j dx\le 0.
 \end{align*} 
\end{proof}

\subsection{Compactness}
In this subsection we give a necessary condition for compactness of the resolvent of the operator $A$ in $L^2(\R^d,\R^m)$ and we give counter example when the condition is not satisfied. Our assumption is that the smallest eigenvalue $\mu(x)$ of $V(x)$ blow up at infinity, which we rewrite as follow:\\
 There there exists $\mu:\R^d\to\R^+$ locally integrable such that $\displaystyle\lim_{|x|\to\infty}\mu(x)=+\infty$, and
\begin{equation}\label{cond for comp.}
\langle V(x)\xi,\xi\rangle\ge \mu(x)|\xi|^2,\qquad\forall \xi\in\R^m,\forall x\in\R^d.
\end{equation}
\begin{prop}\label{prop compact}
Assume that \eqref{cond for comp.} is satisfied. Then, $T_p(t)$ is compact in $L^p(\R^d,\C^m)$, for every $t>0$. Consequently, the spectrum of $A_p$ is independent of $p\in(1,\infty)$, countable and consists of negative eigenvalues that accumulate at $-\infty$. 
\end{prop}
\begin{proof}
It suffices to prove $D(a)$ is compactly embedded in $L^2(\R^d,\C^m)$. Indeed, this implies that $A$ has a compact resolvent and, by analycity, $T(t)$ is compact in $L^2(\R^d,\R^m)$, for every $t>0$. The compactness in $L^p(\R^d,\C^m)$, $1<p<\infty$, follows by \cite[Theorem 1.6.1]{Davies} and the $p$-independence of the spectrum by \cite[Corollary 1.6.2]{Davies}.\\ 
Now, let us consider the 'diagonal' sesquilinear form $$a_\mu(f,g)=\int_{\R^d}\sum_{j=1}^{m}\langle Q(x)\nabla f_j(x),\nabla g_j(x)\rangle dx+\int_{\R^d}\sum_{j=1}^{m}\mu(x)f_j(x)g_j(x) dx,$$ with domain $$D(a_\mu)=\{f\in H^1(\R^d,\C^m) :\int_{\R^d} \sum_{j=1}^{m}\mu(x)|f_j(x)|^2 dx < +\infty\}.$$
Since $\displaystyle\lim_{|x|\to\infty}\mu(x)=+\infty$, one has $D(a_\mu)$ is compactly embedded in $L^2(\R^d,\C^m)$, see \cite[Chapter 4]{Davies}. On the other hand, \eqref{cond for comp.} implies $D(a)\subseteq D(a_\mu)$ and $a_\mu(f)\le a(f)$ for all $f\in D(a)$. Thus, $D(a)$ is continuously embedded in $D(a_\mu)$. It follows that the embedding $D(a)\hookrightarrow L^2(\R^d,\C^m)$ is compact. Now, the discretness of the spectrum follows by the spectral mapping theorem, since $A$ has a compact resolvent.
\end{proof}
\begin{examp}
Here we give a counter-example where of Proposition \ref{prop compact} cannot apply and the compactness result fails even if all entries of the matrix potential blow up at infinity. We even have a nonponctual. Let us consider the following two-size matrix-valued function
$$x\mapsto V(x)\coloneqq v(x)\begin{pmatrix}
1 & -1\\
-1& 1\\
\end{pmatrix}=v(x)J,$$
where $v\in L^1_{loc}(\R^d)$ is a nonnegative function such that $\displaystyle\lim_{|x|\to\infty}v(x)=+\infty$.
$V$ is symmetric and satisfies \eqref{accretivity of V}. $V$ can be written as follow
\[ V(x)=P^{-1}\begin{pmatrix}
2 v(x) & 0\\
0& 0\\
\end{pmatrix}P,\] where $P\coloneqq\begin{pmatrix}
1 & 1\\
-1& 1\\
\end{pmatrix}$. The Schr\"odinger operator $A$ with $Q=I_2$ becomes
\[ A=\Delta-V=P^{-1}\begin{pmatrix}
\Delta-2 v(x) & 0\\
0& \Delta\\
\end{pmatrix}P. \]
Since the Laplacian operator $\Delta$ has no compact resolvent on $L^2(\R^d)$, thus the matrix operator $$\begin{pmatrix}
\Delta-2 v(x) & 0\\
0& \Delta\\
\end{pmatrix}$$ has no compact resolvent. Then so is $A$.\\
Furthermore, the spectrum of $A$ is continuous. In fact, $\sigma(A)=\sigma(\Delta)\cup\sigma(\Delta-2v)=]-\infty,0]$. However, the ponctual spectrum $\sigma_p(A)=\sigma(\Delta-2v)$ is countable.

Such potentials can be constracted even for higher dimensions : $m\ge  3$. One can consider $V(x)=v(x)J_m$ where $v$ is any nonnegative locally integrable function that blow up at infinity and $J_m$ a symmertic semi-definite positive $(m\times m)$-matrix having $0$ as eigenvalue. For instance, one can choose \[J_m=\begin{pmatrix}
  m-1 & -1& \cdots & -1\\
  -1 & m-1 & \ddots &\vdots\\
  \vdots& \ddots& \ddots&-1\\
  -1&\cdots& -1& m-1\\
  \end{pmatrix}.\]
\end{examp}
 \begin{rmk}
 Under the condition \eqref{cond for comp.} which guaranties compactness of the resolvent of $A$, one can get more information about the spectrum $\sigma(A)$ of $A$ by application of the min-max principle. Indeed, let $\mu,\nu:\R^d\to\R^+$ be locally integrables such that $\mu$ blows up at infinity and
 \begin{equation}
 \mu(x)|\xi|^2\le \langle V(x)\xi,\xi\rangle\le \nu(x)|\xi|^2,
 \end{equation}
 for every $x\in\R^d$ and $\xi\in\R^m$. Denotes by $\{\lambda_1<\lambda_2<\dots\}$ the increasing sequence of eigenvalues of $-A$. By $\{\lambda^\mu_1<\lambda^\mu_2<\dots\}$ we denote the eigenvalues of the scalar operator $-(div(Q\nabla\cdot)-\mu)$ in $L^2(\R^d)$. We use the same notation for $\nu$. According to the min-max principle, one has $\lambda_n^\nu\le \lambda_n\le\lambda_n^\mu$, for all $n\in\N$.
 
  We recall that the min-max principle is a way to express eigenvalues of an operator via its associated form, see \cite[Chapter IV]{CH52}. The min-max formula applied to $A$ yields
 \[\lambda_n=\underset{F_1,\dots,F_{n-1}\in H}{\max}\inf\{a(f): f\in\{F_1,\dots,F_{n-1}\}^\bot\cap D(a)\;\emph{with}\;\|f\|=1 \}.\]
 \end{rmk}

 \section*{Acknoledgement}
 The author would like to thank Markus Kunze for suggesting the reference \cite{Ouhabaz 99} which contains the vectorial Denny--Beurling criterion of $L^\infty$--contractivity. He is also grateful to the referee for valuable comments and suggestions.

\end{document}